\documentclass[10pt, reqno]{amsart}
\usepackage{mathrsfs}
\usepackage{epsfig}
\usepackage[dvipsnames]{xcolor}
\usepackage{graphicx}
\usepackage{amsmath}
\usepackage{amsfonts}
\usepackage{amssymb}
\usepackage{bbm}
\usepackage{amsthm}
\usepackage{amscd}
\usepackage{verbatim}
\usepackage{tikz}
\usepackage[sc]{mathpazo}
\usepackage[T1]{fontenc} 
\usepackage[autostyle]{csquotes}

\usepackage{mathtools}
\mathtoolsset{showonlyrefs=true}

\theoremstyle{plain}
\newtheorem{theorem}{Theorem}[section]
\newtheorem{proposition}[theorem]{Proposition}
\newtheorem{lemma}[theorem]{Lemma}

\theoremstyle{definition}

\newtheorem{problem}[theorem]{Problem}

\newtheorem{remark}[theorem]{Remark}
\newtheorem{example}[theorem]{Example}

\numberwithin{equation}{section}

\DeclareMathAlphabet{\mathpzc}{OT1}{pzc}{m}{it}

\newcommand{\p}{\mathpzc{p}}

\newcommand{\tr}{\operatorname{tr}}

\newcommand{\norm}[1]{\| #1 \|}

\let\oldenumerate=\enumerate
	\def\enumerate{
	\oldenumerate
	\setlength{\itemsep}{5pt}
	}
\let\olditemize=\itemize
	\def\itemize{
	\olditemize
	\setlength{\itemsep}{5pt}
	}

\usepackage{hyperref}

\begin{document}

\title[Limiting Distributions of Sums  with Random Spectral Weights]{Limiting Distributions of Sums  with Random Spectral Weights}
\author[\'A. Ch\'avez]{\'Angel Ch\'avez}
\address{Department of Mathematics and Statistics, Pomona College, 610 N. College Ave., Claremont, CA 91711} 
	\email{angel.chavez@pomona.edu}
\author[J. Waldor]{Jacob Waldor}
	\email{jcwa2017@mymail.pomona.edu}

\begin{abstract}
This paper studies the asymptotic properties of weighted sums of the form $Z_n=\sum_{i=1}^n a_i X_i$, in which $X_1, X_2, \ldots, X_n$ are i.i.d.~random variables and $a_1, a_2, \ldots, a_n$ correspond to either eigenvalues or singular values in the classic Erd\H{o}s-R\'enyi-Gilbert model. In particular, we prove central limit-type theorems for the  sequences $n^{-1}Z_n$ with varying conditions imposed on $X_1, X_2, \ldots, X_n$.
\end{abstract}


\maketitle

\section{Introduction}

Suppose $X_1, X_2, \ldots$ is a sequence of random variables. A classic problem in probability is to understand the limiting behavior of the sum
\begin{align*}
Y_n=X_1+X_2+\cdots +X_n.
\end{align*} Early versions of this problem, which consider the case when $X_1, X_2, \ldots$ are independent Bernoulli trials, are rooted in the  work of  de Moivre  \cite{deMoivre}  and Laplace \cite{Laplace} pertaining to normal approximations to the binomial distribution. These classic results mark the beginnings of a long standing problem of approximating laws of sums of random variables by normal distributions. This two-hundred year problem ultimately culminated with what is today known as the central limit theorem. Major contributions towards our modern understanding of the central limit theorem are attributed, in particular,  to  L\'evy \cite{Levy}, Lindeberg \cite{Lindeberg} and Lyapunov \cite{Lyapunov}, among several others.

Suppose $X_1, X_2, \ldots$ are independent and identically distributed (i.i.d.)~random variables with mean $\mu$ and finite positive variance $\sigma^2$. The version of the central limit theorem attributed to L\'evy and Lindeberg \cite[Thm.~27.1]{Billingsley}  asserts that the sequence of random variables $(Y_n-n\mu)/(\sigma\sqrt{n})$ converge in distribution to a standard normal. A theorem attributed to Lyapunov \cite[Thm.~27.3]{Billingsley} shows the assumption that the variables $X_1, X_2, \ldots$ be identically distributed can even be relaxed as long as the absolute moments of the $X_i$ satisfy a certain (Lyapunov) growth condition.

The present focus is to study the limiting behavior of sequences of the form
\begin{align}
Z_n(\boldsymbol{a})=a_1X_1+a_2X_2+\cdots +a_nX_n,\label{WeightedSum}
\end{align} in which $X_1, X_2, \ldots $ are i.i.d.~ random variables and the weights $a_1, a_2,\ldots $ correspond to either the eigenvalues or the singular values of a random symmetric matrix. Specifically, we take eigenvalues and singular values corresponding to the Erd\H{o}s-R\'enyi-Gilbert random graph model. A random graph in this model, which was developed independently by Erd\H{o}s-R\'enyi \cite{Erdos1,Erdos2} and Gilbert \cite{Gilbert},  is constructed by attaching edges among a set of labeled vertices independently with probability $p$. The random variables $a_iX_i$ in this case are neither independent nor indentically distributed, and there is no general method available to handle this situation. However, adjacency matrices of Erd\H{o}s-R\'enyi-Gilbert graphs have bounded entries which, modulo the constraints imposed by symmetry, are independent. This simple fact, together, with the almost sure convergence of their empirical spectral distributions to the semicircular law,  allow us to establish central limit-type theorems for the sequences $n^{-1}Z_n(\boldsymbol{a})$.

\subsection{Notation and Terminology}

 Graph theoretic terminology may be found in \cite{Bollobas}. A \emph{graph $G$ of order $n$}  is an ordered pair $(V, E)$ consisting of a set $E=E(G)$ of \emph{edges} and a set $V=V(G)$ of \emph{vertices} such that $|V|=n$. We adopt standard notation and let $m=|E(G)|$ denote the number of edges in a graph $G$. A graph is \emph{simple} if it contains no loops or multiple edges and it is \emph{connected} if it contains no isolated vertices. If $G$ is a graph of order $n$, then its \emph{adjacency matrix} is the $n\times n$ real symmetric matrix  $A(G)$ whose entries are defined by setting $[A(G)]_{ij}=1$ if vertices $i$ and $j$ are connected by an edge and $[A(G)]_{ij}= 0$ otherwise.  The \emph{spectrum} of $G$ is the spectrum of its adjacency matrix $A(G)$, and is therefore  real since $A(G)$ is Hermitian. We adopt a standard convention and write the spectrum of $G$ in non-increasing order,
\begin{align*}
\lambda_1\geq \lambda_2\geq \cdots \geq \lambda_n.
\end{align*} Moreover, the \emph{singular spectrum} of $G$ consists of the singular values of $A(G)$. We remark that the singular values of the Hermitian matrix $A(G)$ correspond to the moduli of its eigenvalues. Again, we adopt a standard convention and write the singular spectrum of $G$ in non-increasing order,
\begin{align*}
s_1\geq s_2\geq \cdots \geq s_n.
\end{align*}

For $q\geq 1$, we let $L^q(\Omega, \mathcal{F}, \mathbb{P}_{\Omega})$ denote the vector space of random variables defined on the probability space $(\Omega, \mathcal{F}, \mathbb{P}_{\Omega})$ with finite $L^q$-norm defined by
\begin{align*}
\norm{X}_q=\Big(\mathbb{E}|X|^q\Big)^{1/q}.
\end{align*} A random variable $X$ defined on a probability space $(\Omega, \mathcal{F}, \mathbb{P}_{\Omega})$ is called \emph{sub-gaussian} if $\norm{X}_q\leq \norm{X}_{\psi_2}\sqrt{q}$ for all $q\geq 1$, where
\begin{align*}
\norm{X}_{\psi_2}=\sup_{q\geq 1} \Big\{ q^{-1/2}\norm{X}_q\Big\}
\end{align*} is called the \emph{sub-gaussian norm of} $X$ \cite[Def.~2.5.6]{Vershynin}.  Gaussian, Bernoulli and bounded random variables are typical examples of sub-gaussian random variables \cite[Ex.~2.5.8]{Vershynin}.

\subsection{Statement of Results}

This paper establishes central limit-type theorems for the sequences of weighted sums
\begin{align}
W_n(\boldsymbol{a})=\frac{1}{n}\sum_{j=1}^n a_jX_j,\label{EigenSequence}
\end{align} in which $X_1, X_2, \ldots, X_n$ are i.i.d.~random variables and $a_1, a_2, \ldots, a_n$ correspond to either the eigenvalues or the singular values of Erd\H{o}s-R\'enyi-Gilbert graphs. Graphs in the \emph{Erd\H{o}s-R\'enyi-Gilbert $\mathcal{G}(n,p)$ model} are constructed by attaching edges between each vertex pair from a set of $n$ labeled vertices independently with probability $p\in [0,1]$. Here and henceforth we assume $p\neq 0,1$.

 The first two theorems illustrate the relatively simple  limiting distributions of $W_n(\boldsymbol{\lambda})$ in the case when certain symmetry conditions are imposed on $X_1, X_2, \ldots$. The third theorem illustrates the simple limiting distributions of $W_n(\boldsymbol{s})$ in the case when $X_1,X_2,\ldots$ are sub-gaussian with mean zero but not necessarily symmetric.

\begin{theorem}\label{MainTheorem1}
Suppose $X$ is a normal random variable with mean $\mu$ and variance $\sigma^2$.  If $X_1, X_2, \ldots, X_n\sim X$ are i.i.d.~ random variables, then
$W_n(\boldsymbol{\lambda})/(\sigma\sqrt{p})$ converges in distribution to a standard normal.
\end{theorem}

\begin{theorem}\label{MainTheorem2}
Suppose $X$ is a symmetric sub-gaussian random variable with variance $\sigma^2$. If $X_1, X_2, \ldots, X_n\sim X$ are i.i.d.~random variables, then
\begin{align*}
W_n(\boldsymbol{\lambda})\to pX+N_p
\end{align*} in distribution, where $N_p$ is a normal random variable, independent of $pX$, with mean zero and variance $p(1-p)\sigma^2$.
\end{theorem}

\begin{theorem}\label{MainTheorem3}
Suppose $X$ is a sub-gaussian random variable with mean zero and variance $\sigma^2$. If $X_1, X_2, \ldots, X_n\sim X$ are i.i.d.~random variables, then
\begin{align*}
 W_n(\boldsymbol{s})\to pX+N_p
\end{align*} in distribution, where $N_p$ is a normal random variable, independent of $pX$, with mean zero and variance $p(1-p)\sigma^2$.
\end{theorem}

The final theorem illustrates, in particular, how sensitive the sequences $W_n(\boldsymbol{s})$ are with respect to the conditions imposed on $X_1, X_2, \ldots, X_n$. Very distinct behavior emerges by simply choosing random variables with non-zero mean. Ultimately, this distinction is rooted in the asymptotic behavior of the \emph{Schatten $q$ norm} of a graph, which is defined for $q\geq 1$ by
\begin{align*}
\norm{G}_{S_q}=\Big( s_1^q+s_2^q+\cdots + s_n^q\Big)^{1/q}.
\end{align*} In particular, the large $n$ behavior of $\norm{G}_{S_q}$ is very different for the cases $q=1$ and $q>1$ \cite[Thm.~5]{Nikiforov}. Note the sub-gaussian condition is also removed in the following theorem.

\begin{theorem}\label{MainTheorem4}
Suppose $X$ is any random variable with non-zero mean $\mu$ which admits a moment generating function. If $X_1, X_2, \ldots, X_n\sim X$ are i.i.d.~random variables, then $W_n(\boldsymbol{s})/(\mu\sqrt{n})$ converges in distribution to a point mass at $\frac{8}{3\pi}\sqrt{p(1-p)}$.
\end{theorem}

There exist various central limit-type theorems in the literature pertaining to sums of eigenvalues of random matrices e.g.~\cite{Cipolloni, Johansson, Lytova, Shcherbina}. Among the earliest results in this direction are due to Johansson \cite{Johansson}. These specific results concern random Hermitian matrices distributed according to the probability measure
\begin{align}
d\mu_{n,\tau}(A) \propto\exp\big( -\tau\tr V(A)\big)\,dA,\label{HermitianMeasure}
\end{align} in which $\tau>0$, $V(x)$ is a polynomial with positive even degree and a positive leading coefficient and $dA$ denotes Lebesgue measure on the space of $n\times n$ complex Hermitian matrices. It bears worth mentioning that setting $\tau=n/2$ and $V(x)=x^2$ gives rise to the Gaussian unitary ensemble introduced by Dyson \cite{Dyson1, Dyson2}. The main result of \cite{Johansson} establishes that the linear eigenvalue statistic $\sum_{k} f(\lambda_k)$ converges in distribution to a normal random variable with mean zero. The sums we consider in this paper can be thought of as randomized graph-theoretic versions of the sums originally considered by Johansson.

\subsection{Examples and Simulations}

The following examples and simulations illustrate a few of the main theorems. All code used to generate these plots is made available by contacting the authors.

\begin{example}
Suppose $X$ is a normal random variable with mean $\mu$ and variance $\sigma^2$. Theorem \ref{MainTheorem1} ensures that $W_n(\boldsymbol{\lambda})/(\sigma\sqrt{p})$ converges in distribution to a standard normal. The simulation in Figure \ref{fig:Example1} is performed with $\mu=\sigma^2=1$.
\end{example}

\begin{figure}[h!]
  \includegraphics[scale=0.45]{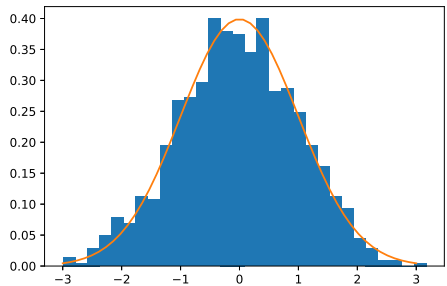}
  \caption{Histogram plot for $W_n(\boldsymbol{\lambda})/(\sigma\sqrt{p})$ using 750 trials taken with $n=1000$ and $p=1/2$.}
  \label{fig:Example1}
\end{figure}

\begin{example}
Suppose $X$ is a Rademacher random variable. In particular, one has $\mathbb{P}(X=1)=\mathbb{P}(X=-1)=1/2$. The random variable $X$ is sub-gaussian and Theorem \ref{MainTheorem2} ensures that $W_n(\boldsymbol{\lambda})$ converges in distribution to the sum of $pX$ with an independent normal random variable with mean zero and variance $\sigma_p^2=p(1-p)$. The probability density of this sum is given by the convolution of the gaussian $f(x)=(\sigma_p\sqrt{2\pi})^{-1}\exp\big[ -x^2/(2\sigma_p^2)\big]$ with $\big[  \delta(x+p) + \delta(x-p)\big]/2 $. Interestingly, this density corresponds to the gaussian mixture $\big[f(x+p)+f(x-p)\big]/2$, which is bimodal in the case $p>1/2$. Figure \ref{fig:Example2} shows histogram plots for $W_n(\boldsymbol{\lambda})$ in the cases $p=1/2$ and $p=3/4$.
\end{example}

\begin{figure}[h!]
  \includegraphics[scale=0.4]{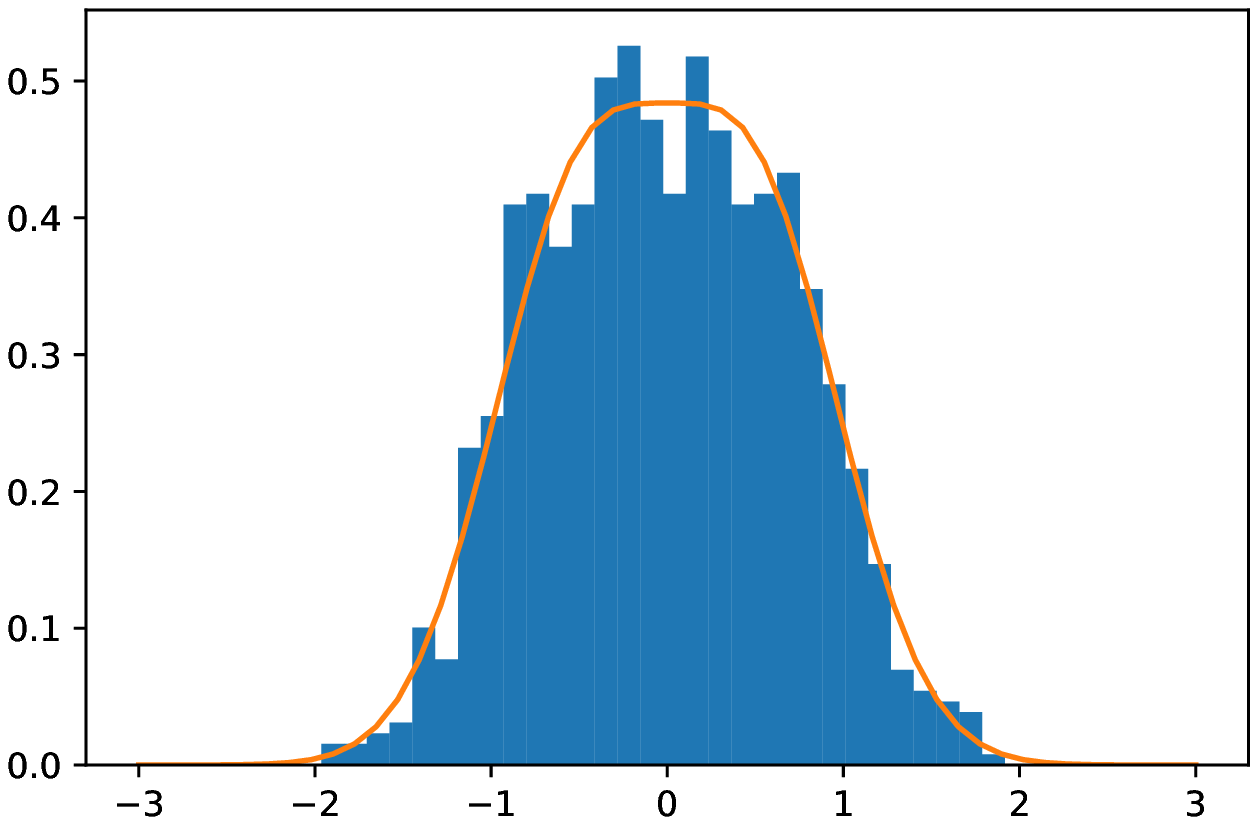}  \includegraphics[scale=0.4]{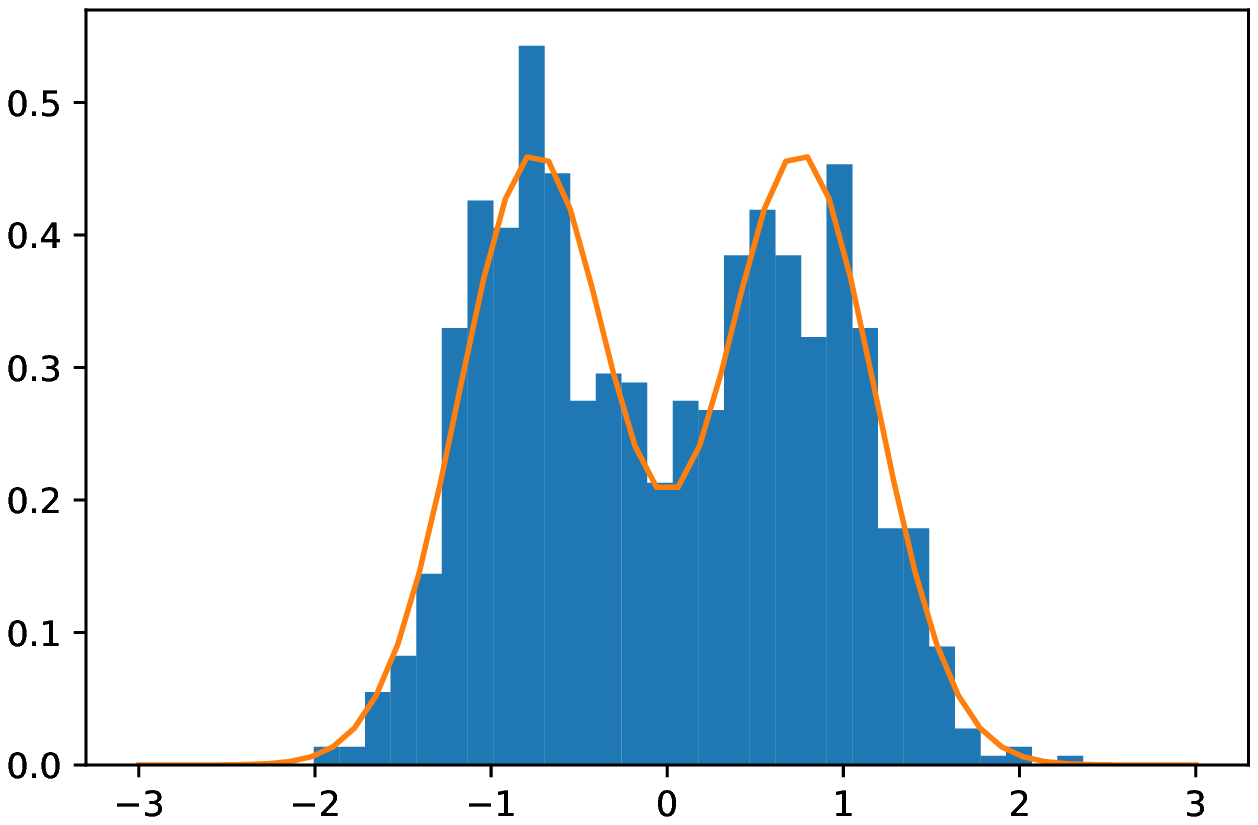}
  \caption{Histogram plots for $W_n(\boldsymbol{\lambda})$ using 750 trials taken with $n=1000$ and $p=1/2$ (left) and $p=3/4$ (right). }
  \label{fig:Example2}
\end{figure}

\begin{example}
Suppose $X$ is a normal random variable with non-zero mean $\mu$ and variance $\sigma^2$. Theorem \ref{MainTheorem4} ensures that $W_n(\boldsymbol{s})/(\mu\sqrt{n})$ converges in distribution to a point mass at $\frac{8}{3\pi} \sqrt{p(1-p)}$. The simulation in Figure \ref{fig:Example3} is performed with $\mu=\sigma^2=1$.

\begin{figure}[h!]
\centering
  \includegraphics[scale=0.45]{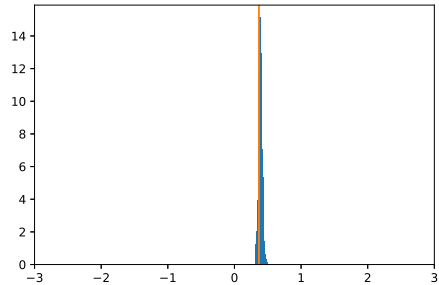}
  \caption{Histogram plot for $W_n(\boldsymbol{s})/(\mu\sqrt{n})$ using 750 trials  with $n=1000$ and $p=3/4$. }
  \label{fig:Example3}
\end{figure}
\end{example}

\subsection{Outline of the Paper}
This paper, which is intended for a wide probabilistic audience, takes us on a short journey through the spectral analysis of large random matrices and is organized as follows. Section \ref{Section2} highlights a few classic results in random matrix theory which serve as prerequisites for later sections. No background in random matrices is assumed. Section \ref{Section3} provides a computational lemma that we use to expand the partial moments of $W_n(\boldsymbol{a})$ in terms of power sum symmetric functions.  Section \ref{Section4} establishes the asymptotics for the partial moments of $W_n(\boldsymbol{a})$ by analyzing the limiting behavior of the power sum symmetric functions. Theorems \ref{MainTheorem1} and \ref{MainTheorem2}  are proved in Section \ref{Section5}. Theorem \ref{MainTheorem3} is proved in Section \ref{Section6} and Theorem \ref{MainTheorem4} is proved in Section \ref{Section7}.   Finally, we conclude with possible directions for future work and closing remarks.


\section{Random Matrix Prerequisites}\label{Section2}

 The limiting spectral analysis for large random matrices has become a widely studied topic in probability since the pioneering work of Eugene Wigner who  proved that the expected empirical spectral distribution of a normalized $n\times n$ (Wigner) matrix tends to the semicircular law $\mu_{\operatorname{sc}}$. To begin, suppose $A$ is an $n\times n$ Hermitian matrix with complex entries. The eigenvalues $\lambda_1, \lambda_2, \ldots, \lambda_n$ of $A$ are real and we can define
the one-dimensional distribution function
\begin{align*}
\mu_{A}(x)=\frac{1}{n}\Big\vert\{ i\leq n\,:\,\lambda_i\leq x\}\Big\vert
\end{align*} called the  \emph{empircal spectral distribution (ESD)} of $A$. The relation  \cite[Sec.~1.3.1]{Bai} 
\begin{align}
\frac{1}{n} \tr(A^k)=\int_{-\infty}^{\infty} x^k\, d\mu_A(x)\label{ESDTrace}
\end{align} plays a fundamental role in random matrix theory. Specifically, it turns the problem of establishing convergence, in whatever sense, for the ESD of a sequence $\{A_n\}$ of random matrices into the problem of establishing convergence of the sequence $\{\frac{1}{n}\tr(A_n^k)\}$ for each fixed $k$.

An \emph{$n\times n$ symmetric Wigner matrix} is an $n\times n$ real symmetric matrix whose entries, modulo the symmetry condition $\xi_{ij}=\xi_{ji}$, are independent. Specifically, we permit i.i.d.~mean zero entries $\xi_{ij}$ above the main diagonal and i.i.d.~mean zero entries $\xi_{ii}$ on the main diagonal. These two families need not share the same distribution, however. Moreover, we impose the condition that all entries have bounded moments and share a common second moment. If $B_n$ is an $n\times n$ symmetric Wigner matrix, then we denote $A_n=B_n/\sqrt{n}$. The pioneering work of Wigner \cite{Wigner1, Wigner2} establishes 
\begin{align}
\lim_{n\to \infty} \frac{1}{n}\mathbb{E} \tr(A_n^k)=\int_{-\infty}^{\infty} x^k \,d\mu_{\operatorname{sc}}(x) \label{Wigner}
\end{align} for all integers $k\geq 1$. In particular, the expected ESD of a normalized $n\times n$ symmetric Wigner matrix tends to the semicircular law whose density is given by 
\begin{align*}
f_{\operatorname{sc}}(x)=\frac{1}{2\pi}\sqrt{4-x^2} \,\mathbb{1}_{[-2,2]}.
\end{align*} This original result due to Wigner has been extended in several aspects. Grenander \cite{Grenander} proved the empirical spectral distribution converges to $\mu_{\operatorname{sc}}$ in probability. Arnold \cite{Arnold1, Arnold2} further improved this result by showing the empirical spectral distribution converges to the semicircular law almost surely. We remark that the matrix ensembles underlying \eqref{Wigner} can be generalized beyond those originally considered by Wigner and refer the reader to \cite{Bai} and \cite{Tao} for excellent surveys on the rich and rapidly developing field of spectral analysis of large random matrices.

\subsection{Almost Sure Convergence of the ESD}\label{Section3.1}

The form of \eqref{Wigner} that we need for later sections is due to Arnold. In particular, suppose $B_n$ is an $n \times n$ real symmetric matrix whose entries, modulo the symmetry condition $\xi_{ij}=\xi_{ji}$, are independent. Assume the upper-triangular entries share a common distribution with finite positive variance $\sigma^2$. In addition, we assume the diagonal entries also share a common distribution.  Furthermore, assume the entries $\xi_{ij}$ have finite fourth and sixth moments for $i>j$  and the diagonal entries $\xi_{ii}$ have  finite second and fourth moments. Define the normalized matrix $A_n=B_n/ (\sigma \sqrt{n})$.  Arnold proves that $\mu_{A_n}\to \mu_{\operatorname{sc}}$ almost surely in the sense that
\begin{align}
\int_{\mathbb{R}} \phi (x)\,d\mu_{A_n}(x)\to \int_{\mathbb{R}} \phi (x)\,d\mu_{\operatorname{sc}}(x)\label{Arnold}
\end{align} for all continuous and compactly supported test functions $\phi:\mathbb{R}\to \mathbb{R}$    \cite[Thm.~2]{Arnold2}.

\subsection{Real Symmetric Matrices with Independent Bounded Entries}\label{Section3.2}

 Here and throughout we adopt standard asymptotic notation. In particular, we write $f(n)=O\big( g(n)\big)$ if there exists a constant $C$ such  that $|f(n)|\leq Cg(n)$ for all $n$ sufficiently large. Moreover, we write $f(n)=o\big( g(n)\big)$ whenever $f(n)/g(n)\to 0$ as $n\to \infty$. 

 A result due to F\"uredi and Koml\'os allows us to analyze the limiting behavior of polynomials in $a_1, a_2, \ldots, a_n$. Suppose $A_n$ is an $n\times n$ real symmetric matrix with bounded entries. Moreover, we assume the entries of $A_n$, modulo the constraint $\xi_{ij}=\xi_{ji}$, are independent. Let $\mu=\mathbb{E}\xi_{ij}>0$ denote the common mean of the upper-triangular entries and let $\sigma^2=\mathbb{E}(\xi_{ij}-\mu)^2$ denote their common variance. Furthermore, suppose the diagonal entries share a common mean, $\nu=\mathbb{E}\xi_{ii}$.  F\"uredi and Koml\'os \cite[Thm.~1]{Furedi} show that the distribution of the largest eigenvalue $\lambda_1$ of $A_n$ can be approximated in order $n^{-1/2}$ by a normal distribution with mean $(n-1)\mu+\nu+\sigma^2/\mu$ and variance $2\sigma^2$. Moreover,  with high probability (w.h.p.) we have
\begin{align*}
|\lambda_i|<\sigma\sqrt{n}+O(n^{1/3}\log n)
\end{align*} whenever $i\neq 1$.


\section{A Computational Lemma for the Partial Moments of $W_n(\boldsymbol{a})$}\label{Section3}

Suppose $j\geq 1$ and $X_1, X_2, \ldots, X_n$ are i.i.d.~ random variables defined on the probability space $(\Omega, \mathcal{F}, \mathbb{P}_{\Omega})$. The lemma we present is a simple, albeit useful, computational tool for evaluating the moments
\begin{align*}
f_j(\boldsymbol{a})=f_j^{(n)}(\boldsymbol{a})=\mathbb{E}_{\Omega}\big( a_1X_1+a_2X_2+\cdots+a_nX_n)^j,
\end{align*} in which $\mathbb{E}_{\Omega}$ denotes expectation with respect to $X_1, X_2, \ldots, X_n$.  This lemma expresses $f_j(\boldsymbol{a})$ as a sum taken over all partitions of $j$ and involves  power sum symmetric polynomials in the variables $a_1, a_2, \ldots, a_n$. We recall these definitions below and refer the reader to \cite[Sec.~1.7]{Stanley1} and \cite[Sec.~7.7]{Stanley2} for in depth discussions.

A \emph{partition} of an integer $j\geq 1$ is a non-increasing sequence $\boldsymbol{\pi}=(j_1, j_2, \ldots, j_r)$ of positive integers such that $j_1+j_2+\cdots+j_r=j$. If $j\geq 2$ is an even integer, then a partition $\boldsymbol{\pi}=(j_1, j_2, \ldots, j_r)$  is a \textit{partition into even} parts if $j_1, j_2, \ldots, j_r$ are even integers. We let $P(j)$ and $E(j)$ denote the set of all partitions of $j$ and the set of all partitions of $j$ into even parts, respectively. We define 
\begin{align*}
y_{\boldsymbol{\pi}}=\prod_{i\geq 1}(i!)^{m_i}m_i!,
\end{align*} in which $m_i=m_i(\boldsymbol{\pi})$ denotes the multiplicity of $i$ appearing in a partition $\boldsymbol{\pi}$. Lastly, the \emph{power sum symmetric polynomial of degree $j$} in the variables $a_1, a_2, \ldots, a_n$ is the homogeneous polynomial defined by setting
\begin{align*}
\p_j(\boldsymbol{a})=\p_j^{(n)}(\boldsymbol{a})=a_1^j+a_2^j+\cdots+a_n^j.
\end{align*} We often denote $\p_j=\p_j(\boldsymbol{a})$ for brevity when there is no risk of confusion.

\begin{lemma}\label{Lemma:Tracial}
Let $j\geq 1$ be any integer and suppose $X_1, X_2, \ldots, X_n$ are i.i.d.~random variables which admit a moment generating function. If $\kappa_1, \kappa_2,\ldots$ denotes the cumulants of the $X_i$, then
\begin{align*}
f_j(\boldsymbol{a})=j!\sum_{\boldsymbol{\pi}\in P(j)} \frac{ \kappa_{\boldsymbol{\pi}}}{y_{\boldsymbol{\pi}}}  \p_{\boldsymbol{\pi}},
\end{align*} where $\kappa_{\boldsymbol{\pi}}=\kappa_{j_1}\kappa_{j_2}\cdots \kappa_{j_r}$ and $\p_{\boldsymbol{\pi}}=\p_{j_1}\p_{j_2}\cdots \p_{j_r}$ given $\boldsymbol{\pi}=(j_1, j_2, \ldots, j_r)$.
\end{lemma}

\begin{proof}
The random variables $X_1, X_2, \ldots, X_n$ are i.i.d.~which implies that the moment generating function of $\sum_{k=1}^n a_kX_k$ takes  the form $\prod_{k=1}^n M(a_kt)$, where $M(t)$ denotes the moment generating function of the $X_i$ \cite[Sec.~9]{Billingsley}. Therefore,
\begin{align*}
\sum_{j=0}^{\infty} f_j(\boldsymbol{a}) \frac{t^j}{j!}&=\prod_{k=1}^n M(a_kt)=\exp\Big( K(a_1t)+K(a_2t)+\cdots +K(a_nt)\Big),
\end{align*} in which $K(t)=\log M(t)$ denotes the cumulant generating function of the $X_i$. The identity $K(t)=\sum_{\ell=1}^{\infty} \kappa_{\ell} t^{\ell}/\ell!$, which defines the cumulant sequence $\kappa_1, \kappa_2, \ldots$, implies
\begin{align}
\sum_{j=0}^{\infty} f_j(\boldsymbol{a}) \frac{t^j}{j!}=\exp\Big(\sum_{\ell=1}^{\infty} \kappa_{\ell}\p_{\ell} \frac{t^{\ell}}{\ell !} \Big)=\sum_{j=0}^{\infty} B_j( \kappa_1 \p_1, \kappa_2 \p_2, \ldots, \kappa_j\p_j) \frac{t^j}{j!},\label{Nbhd}
\end{align} where $B_j(x_1, \ldots, x_j)$ denotes the \emph{complete Bell polynomial} of degree $j$ in the variables $x_1, x_2, \ldots, x_j$ \cite[Sec.~II]{Bell} defined via the generating function
\begin{align}
\sum_{j=0}^{\infty} B_j( x_1, x_2, \ldots, x_j) \frac{t^j}{j!}=\exp\Big(\sum_{\ell=1}^{\infty} x_{\ell} \frac{t^{\ell}}{\ell !} \Big).\label{BellGen}
\end{align} Comparing coefficients in the above expression and applying the identity
\begin{align}
B_j(x_1, x_2, \ldots, x_j)=j !\sum_{\substack{i_1,i_2, \ldots, i_{j}\geq 0\\ i_1+2i_2+\cdots +j i_j=j  }}\prod_{s=1}^{j} \frac{x_s^{i_s}}{(s!)^{i_s} i_s!}=j!\sum_{\boldsymbol{\pi}\in P(j)} \frac{x_{\boldsymbol{\pi}}}{y_{\boldsymbol{\pi}}}\label{BellPartition}
\end{align} completes the proof.
\end{proof}

\begin{remark}
$X$ admits a moment generating function, which implies that its cumulant generating function $K(t)$ converges in a neighborhood of $t=0$ \cite[Sec.~9]{Billingsley}. Relation \eqref{Nbhd} therefore holds in a neighborhood of $t=0$. We use this fact repeatedly throughout the rest of the paper.
\end{remark}


\section{Asymptotics for Erd\H{o}s-R\'enyi-Gilbert Graphs}\label{Section4}

Here and throughout we let $\mathbb{E}_{\mathcal{G}}=\mathbb{E}_{\mathcal{G}(n,p)}$ and $\mathbb{P}_{\mathcal{G}}=\mathbb{P}_{\mathcal{G}(n,p)}$ denote expectation and probability, respectively, with respect to $\mathcal{G}(n,p)$. A fundamental fact is that the number of edges in a random graph of order $n$ follow a binomial distribution,
\begin{align*}
\mathbb{P}_{\mathcal{G}}\big(  |E(G)|=m \big)={N\choose m}p^m(1-p)^{N-m},
\end{align*} where $N={n\choose 2}$. Moreover, the adjacency matrix $A_n$ of random graph of order $n$ is a random $n\times n$ real symmetric matrix whose upper-triangular entries $\xi_{ij}$ are bounded independent random variables which have mean $\mu=p$ and variance $\sigma_p^2=p(1-p)$. The diagonal elements of $A_n$ satisfy $\nu=\mathbb{E}_{\mathcal{G}}[\xi_{ii}]=0$ since loops are not permitted. The result by  F\"uredi and Koml\'os \cite[Thm.~1]{Furedi}, which we outline in Section \ref{Section3.2},  implies that w.h.p.,
\begin{align}
|\lambda_1|=(n-1)p+1-p+O(n^{-1/2})=\big( p+o(1)\big)n\label{Erdos1}
\end{align} and
\begin{align}
|\lambda_2|< 2\sigma_p\sqrt{n}+O(n^{1/3}\log n)=\Big( 2\sigma_p+o(1)\Big) \sqrt{n}\label{Erdos2}.
\end{align}

\subsection{The Eigenvalue Case}

We now establish the  limiting behavior for the partial moments $f_j(\boldsymbol{\lambda})=\mathbb{E}_{\Omega}(\lambda_1X_1+\lambda_2X_2+\cdots+\lambda_nX_n)^j$ in which $X_1, X_2, \ldots, X_n$ are i.i.d.~random variables defined on $(\Omega, \mathcal{F}, \mathbb{P}_{\Omega})$ and $\lambda_1, \lambda_2,\ldots, \lambda_n$ correspond to eigenvalues in the Erd\H{o}s-R\'enyi-Gilbert model. We recall $\p_k(\boldsymbol{\lambda})$ denotes the power sum symmetric polynomial of degree $k$ in the variables $\lambda_1, \lambda_2, \ldots, \lambda_n$. 

\begin{lemma}\label{OddTrace}
Let $k\geq 1$ be an odd integer. If $\lambda_1, \lambda_2, \ldots, \lambda_n$ correspond to eigenvalues in the Erd\H{o}s-R\'enyi-Gilbert $\mathcal{G}(n,p)$ model, then  we have $\p_k(\boldsymbol{\lambda})=o(n^{1+k/2})$ almost surely. 
\end{lemma}

\begin{proof}
Let $B_n$ be the adjacency matrix of a random graph of order $n$. This matrix satisfies the hypotheses in Section \ref{Section3.1}. Moreover, the variance of the upper-triangular entries is given by $\sigma_p^2=p(1-p)$. The ESD of the normalized matrix  $A_n=B_n/(\sqrt{n}\sigma_p)$ converges almost surely to the semicircular law by \cite[Thm.~2]{Arnold2} as seen in Section \ref{Section3.1}. Therefore, we can use \eqref{ESDTrace} and the identity  $\tr(B_n^k)=\p_k(\boldsymbol{\lambda})$ to conclude
\begin{align}
\frac{\p_k(\boldsymbol{\lambda})}{n (\sqrt{n}\sigma_p)^k}=\frac{1}{n} \tr(A_n^k)\to \int_{-\infty}^{\infty} x^k \,d\mu_{\operatorname{sc}}(x)\label{ArnoldLimit}
\end{align} almost surely. The symmetry of the the semicircular density and the fact that $k$ is odd imply that $\frac{\p_k(\boldsymbol{\lambda})}{n (\sqrt{n}\sigma_p)^k}\to 0 $ almost surely. The claim follows.
\end{proof}

\begin{proposition}\label{Proposition:AsymptoticEigen}
Let $j\geq 1$ be any integer and let $X_1, X_2, \ldots, X_n$ be i.i.d.~random variables with cumulant sequence $\kappa_1, \kappa_2,\ldots$. Define
\begin{align*}
c_j(p)=j!\sum_{\boldsymbol{\pi}\in E(j)} \frac{\kappa_{\boldsymbol{\pi}}}{y_{\boldsymbol{\pi}}} p^{j-m_2}  ,
\end{align*}where $\kappa_{\boldsymbol{\pi}}$ and $y_{\boldsymbol{\pi}}$ are defined as in Lemma \ref{Lemma:Tracial} and $E(j)$ denotes the set of partitions of $j$ into even parts. The partial moments $f_j(\boldsymbol{\lambda})$ satisfy the following, where $o(1)$ denotes a term tending to zero as $n\to \infty$ with $j$ fixed.\\

\noindent (a) If $j$ is odd, then $n^{-j}f_j(\boldsymbol{\lambda})=o(1)$ w.h.p. \\

\noindent(b) If $j$ is even, then $n^{-j}f_j(\boldsymbol{\lambda})=c_j(p)+o(1)$ w.h.p.
\end{proposition}

\begin{proof}
Denote $\p_i=\p_i(\boldsymbol{\lambda})$ for brevity.  The number of edges in a random graph $G_n$ of order $n$ is a binomial random variable with parameters $N={n\choose 2}$ and $p$. The expected number of edges  in $G_n$ is therefore given by  $\mathbb{E}_{\mathcal{G}}[m]=pN$. The weak law of large numbers  implies that $m$ is  tightly concentrated around its mean for large $n$. Therefore, we have $m=\big( p/2+o(1)\big)n^2$ w.h.p. If $A_n$ denotes the adjaceny matrix of $G_n$, then $\p_2=\tr(A_n^2)=2m$ \cite[Thm.~3.1.1]{Cvetkovic}, which implies that w.h.p.,
\begin{align}
\mathpzc{p}_2=\big(p+o(1)\big)n^2.\label{ErdosBound1}
\end{align}  If $i>2$ is even, then consider the bound $|\lambda_1|^i\leq \p_i\leq |\lambda_1|^i+n|\lambda_2|^i.$
Inequalities \eqref{Erdos1} and \eqref{Erdos2} imply that w.h.p., 
\begin{align*}
\big(p^i+o(1)\big)n^i\leq \mathpzc{p}_i<\big(p^i+o(1)\big)n^i+O(n^{1+i/2}).
\end{align*} Observe that $n^{1-i/2}\to 0$ as $n\to \infty$ since $i>2$. We conclude that w.h.p.,
\begin{align}
\p_i=\big( p^i+o(1)\big)n^i.\label{ErdosBound2}
\end{align} Let $\boldsymbol{\pi}=(j_1,j_2, \ldots, j_r)$ be a partition of $j$. Lemma \ref{OddTrace} together with relations \eqref{ErdosBound1} and \eqref{ErdosBound2} imply that w.h.p,
\begin{align}
\p_{j_1}\p_{j_2}\cdots \p_{j_r}=o(n^j)\label{ErdosOdd}
\end{align} whenever $\boldsymbol{\pi}$ contains an odd integer larger than one. If $m_1(\boldsymbol{\pi})\neq 0$, then \eqref{ErdosOdd} still holds since $\p_1=\tr(A_n)=0$ for simple graphs. Any partition of an odd integer must contain an odd part. Lemma \ref{Lemma:Tracial} implies that w.h.p.,
\begin{align*}
f_j(\boldsymbol{\lambda})=o(n^j)
\end{align*} whenever $j$ is odd. This proves (a).  If $j$ is even, then $\p_{j_1}\p_{j_2}\cdots \p_{j_r}=o(n^j)$ w.h.p.~unless $\boldsymbol{\pi}$ is a partition of $j$ into even parts. Lemma \ref{Lemma:Tracial}, together with relations \eqref{ErdosBound1} and \eqref{ErdosBound2}, imply that w.h.p.,
\begin{align*}
f_j(\boldsymbol{\lambda})=\Bigg( j! \sum_{\boldsymbol{\pi}\in E(d)} \frac{\kappa_{\boldsymbol{\pi}}}{y_{\boldsymbol{\pi}}} p^{j-m_2}+o(1)  \Bigg)n^j,
\end{align*} which proves (b). We remark that the $m_2$ term appearing in the last expression occurs because of the discrepency in the power of $p$ occuring in \eqref{ErdosBound1} and \eqref{ErdosBound2}.
\end{proof}

\subsection{The Singular Value Case}

We now establish the  limiting behavior for the partial moments $f_j(\boldsymbol{s})=\mathbb{E}_{\Omega}(s_1X_1+s_2X_2+\cdots+s_nX_n)^j$ in which $X_1, X_2, \ldots, X_n$ are i.i.d.~random variables defined on $(\Omega, \mathcal{F}, \mathbb{P}_{\Omega})$ and $s_1, s_2,\ldots, s_n$ correspond to singular values in the Erd\H{o}s-R\'enyi-Gilbert model.

\begin{proposition}\label{Proposition:AsymptoticSingular}
Let $j\geq 1$ be any integer and let $X_1, X_2, \ldots, X_n$ be i.i.d.~mean zero random variables with cumulant sequence $\kappa_1, \kappa_2,\ldots$. The partial moments $f_j(\boldsymbol{s})$ satisfy 
\begin{align*}
n^{-j}f_j(\boldsymbol{s})=j!\sum_{\boldsymbol{\pi}\in P_0(j)} \frac{ \kappa_{\boldsymbol{\pi}}}{y_{\boldsymbol{\pi}}}  p^{j-m_2}+o(1)
\end{align*} w.h.p., where $\kappa_{\boldsymbol{\pi}}$ and $y_{\boldsymbol{\pi}}$ are defined as in Lemma \ref{Lemma:Tracial} and $P_0(j)$ denotes the set of all partitions of $j$ for which $m_1=0$.
\end{proposition}

\begin{proof}
The singular values of a random graph correspond to the moduli of its eigenvalues. Therefore,  \eqref{ErdosBound1} and \eqref{ErdosBound2} hold for $\p_i=\p_i(\boldsymbol{s})$ when $i$ is even. The same exact reasoning behind \eqref{ErdosBound2} implies $\p_i=\big(p^i+o(1)\big)n^i$ when $i\neq 1$ is odd. Finally, $\kappa_1=0$ since $X_1, X_2, \ldots, X_n$ have mean zero. Lemma \ref{Lemma:Tracial} implies the claim.
\end{proof}

\begin{proposition}\label{Proposition:AsymptoticEnergy}
Let $j\geq 1$ be any integer and let $X_1, X_2, \ldots, X_n\sim X$ be i.i.d.~random variables in which $X$ admits a moment generating function and has non-zero mean $\mu$. The partial moments $f_j(\boldsymbol{s})$ satisfy
\begin{align*}
\frac{f_j(\boldsymbol{s})}{ (\mu n^{3/2})^j}=\Big(\frac{8\sigma_p}{3\pi}\Big)^j +o(1)
\end{align*} w.h.p., in which  $\sigma_p^2=p(1-p)$.
\end{proposition}

\begin{proof}
The quantity $\p_1=\p_1(\mathbf{s})$ is called the graph energy \cite{Gutman} and \cite[Thm.~1]{Du} establishes that w.h.p., 
\begin{align*}
\frac{\p_1}{n^{3/2}}=\frac{8\sigma_p}{3\pi}+o(1).
\end{align*} Lemma \ref{Lemma:Tracial} asserts
\begin{align}
f_j(\boldsymbol{s})=j!\sum_{\boldsymbol{\pi}\in P(j)} \frac{ \kappa_{\boldsymbol{\pi}}}{y_{\boldsymbol{\pi}}}  \p_{\boldsymbol{\pi}}=\mu^j \p_1^j+j!\sum_{\boldsymbol{\pi}\neq (1,1,\ldots,1)} \frac{ \kappa_{\boldsymbol{\pi}}}{y_{\boldsymbol{\pi}}}  \p_{\boldsymbol{\pi}}\label{EnergyAsymptoticEq}
\end{align} since $\kappa_1=\mu$ and $y_{\boldsymbol{\pi}}=m_1!=j!$ when $\boldsymbol{\pi}=(1,1,\ldots, 1)\in P(j).$ We have so far established that $\p_2=\big(p+o(1)\big)n^2$ and $\p_i=\big(p^i+o(1)\big)n^i$ for integers $i\geq 3$. Dividing both sides of \eqref{EnergyAsymptoticEq} by $(\mu n^{3/2})^j$ implies the claim since $\p_{\boldsymbol{\pi}}/(n^{3/2})^j=o(1)$ w.h.p.~ whenever $\boldsymbol{\pi}\neq (1,1,\ldots,1)$. 
\end{proof}


\section{Proof of Theorems \ref{MainTheorem1} and \ref{MainTheorem2}}\label{Section5}

 The following general form of the Hoeffding inequality \cite[Thm.~2.6.3]{Vershynin} plays an important role in our proof.  Suppose $\xi_1, \xi_2, \ldots, \xi_n$ are independent mean zero sub-gaussian random variables. There exists an absolute constant $c_1>0$ such that
\begin{align*}
\mathbb{P}\Bigg( \Big\vert \sum_{j=1}^n a_i\xi_i  \Big\vert \geq t  \Bigg)\leq 2\exp\Big( - \frac{c_1t^2}{K^2 \p_2(\boldsymbol{a})}   \Big)
\end{align*} for all $a_1, a_2, \ldots, a_n\in \mathbb{R}$, in which $\p_2(\boldsymbol{a})=\sum_i a_i^2$ and $K=\max_i \norm{\xi_i}_{\psi_2}.$ If in addition $\xi_1, \xi_2, \ldots, \xi_n$ have unit variances, then 
\begin{align}
 \norm{\sum_{k=1}^na_iX_i}_1\leq \sqrt{\p_2(\boldsymbol{a})}.\label{Vershynin1}
\end{align} Moreover, there exists an absolute constant $c_2>0$ such that for all $q\geq 2$,
\begin{align}
\norm{\sum_{k=1}^na_iX_i}_q\leq c_2K\sqrt{q} \sqrt{\p_2(\boldsymbol{a})}\label{Vershyninq}.
\end{align}


\subsection{Proof of Theorem \ref{MainTheorem1}}

 Suppose $X_1, X_2, \ldots, X_n\sim X$ are i.i.d.~random variables defined on a probability space $(\Omega, \mathcal{F}_{\Omega}, \mathbb{P}_{\Omega})$  in which  $X$ is a sub-gaussian random variable with mean $\mu$ and variance $\sigma^2$. Define the auxiliary variables $\xi_i=\sigma^{-2}(X_i-\mu)$ and let $j\geq 1$ be any integer. The sum $W_n(\boldsymbol{\lambda})$ satisfies 
\begin{align*}
\mathbb{E}_{\Omega} |W_n(\boldsymbol{\lambda})|^j&=\norm{W_n(\boldsymbol{\lambda})}_j^j \notag \\
&=n^{-j} \norm{\sum_{k=1}^n \lambda_k X_k}_j^j\notag \\
&=n^{-j}\norm{\sum_{k=1}^n \lambda_k( X_k-\mu)+\mu\sum_{k=1}^n\lambda_k}_j^j\notag\\
&=n^{-j}\sigma^{2j} \norm{\sum_{k=1}^n \lambda_k \xi_k}_j^j
\end{align*} since simple graphs are traceless. The random variables $\xi_1, \xi_2, \ldots \xi_n$ are mean zero sub-gaussian random variables with unit variances. Relations \eqref{Vershynin1} and \eqref{Vershyninq}, together with the inequality $2m\leq n(n-1)$, imply that there exists a constant $c>0$, which is independent of $n$, such that
\begin{align}
\mathbb{E}_{\Omega} |W_n(\boldsymbol{\lambda})|^j\leq cn^{-j}\big[ \p_2(\boldsymbol{\lambda})\big]^{j/2}= c n^{-j}(2m)^{j/2}\leq cn^{-j} [n(n-1)]^{j/2}\leq c\label{Uniform}
\end{align} Therefore, $\mathbb{E}_{\mathcal{G}}\mathbb{E}_{\Omega} |W_n(\boldsymbol{\lambda})|^j$ is finite and the Fubini-Tonelli theorem \cite[Thm.~2.16]{Folland} ensures
\begin{align}
\mathbb{E}\big[W_n(\boldsymbol{\lambda})\big]^j=\mathbb{E}_{\mathcal{G}} \mathbb{E}_{\Omega} \big[W_n(\boldsymbol{\lambda})\big]^j,\label{Fubini}
\end{align} in which $\mathbb{E}[\cdot]$ denotes expectation with respect to the product measure $\mathbb{P}_{\mathcal{G}}\otimes \mathbb{P}_{\Omega}$.

Proposition \ref{Proposition:AsymptoticEigen} and the uniform boundedness of the variables $\mathbb{E}_{\Omega} |W_n(\boldsymbol{\lambda})|^j$ allow us to compute the limit for the total expectation of $[W_n(\boldsymbol{\lambda})]^j$, which we now highlight. Define $E_j=\{ G\,:\,  n^{-j}f_j(\boldsymbol{\lambda})=\alpha_j\}$, where $\alpha_j=\alpha_j(n)$ is a real number to be chosen momentarily. Relation \eqref{Fubini} implies
\begin{align*}
\mathbb{E}[W_n(\boldsymbol{\lambda})]^j&=\mathbb{E}_{\mathcal{G}}[n^{-j} f_j(\boldsymbol{\lambda})]\\
&=\int_{\mathcal{G}} n^{-j} f_j(\boldsymbol{\lambda})\,d\mathbb{P}_{\mathcal{G}}\\
&=\alpha_j \mathbb{P}_{\mathcal{G}}( E_j)+\int_{E_j^c}n^{-j} f_j(\boldsymbol{\lambda})\,d\mathbb{P}_{\mathcal{G}}.
\end{align*}  The inequality $n^{-j}|f_j(\boldsymbol{\lambda})|\leq \mathbb{E}_{\Omega} |W_n(\boldsymbol{\lambda})|^j,$ together with \eqref{Uniform}, implies that
\begin{align*}
\Big\vert\int_{E_j^c}n^{-j} f_j(\boldsymbol{\lambda})\,d\mathbb{P}_{\mathcal{G}}\Big\vert\leq \int_{E_j^c}  \mathbb{E}_{\Omega} |W_n(\boldsymbol{\lambda})|^j\,d\mathbb{P}_{\mathcal{G}}\leq c \mathbb{P}_{\mathcal{G}}(E_j^c).
\end{align*} Therefore,
\begin{align*}
\lim_{n\to \infty} \mathbb{E}[W_n(\boldsymbol{\lambda})]^j=\lim_{n\to \infty}\Big( \alpha_j \mathbb{P}_{\mathcal{G}}( E_j)\Big)
\end{align*} provided $\mathbb{P}_{\mathcal{G}}(E_j)\to 1$ as $n\to \infty$ so that $\mathbb{P}_{\mathcal{G}}(E_j^c)\to 0$ as $n\to \infty$.  We now choose $\alpha_j$ according to Proposition \ref{Proposition:AsymptoticEigen} to conclude
\begin{align}
\lim_{n\to \infty} \mathbb{E}\big[W_n(\boldsymbol{\lambda})\big]^j=\begin{cases} 0 & \mbox{for } j\mbox{ odd},\\
c_j(p) & \mbox{for } j \mbox{ even}. \end{cases}\label{MomentMethod1}
\end{align}

A useful criteria for determing when a distribution is determined by its moments is that it admits a moment generating function \cite[Thm.~30.1]{Billingsley}. Suppose that the distribution of a random variable $W$ is determined by its moments. The method of moments \cite[Thm.~30.2]{Billingsley} ensures that $W_n(\boldsymbol{\lambda})$ converges in distribution to $W$ provided $W_n(\boldsymbol{\lambda})$ has moments of all orders and for all $j\geq 1$, 
\begin{align*}
\lim_{n\to \infty} \mathbb{E}\big[W_n(\boldsymbol{\lambda})\big]^j=\mathbb{E}[W^j].
\end{align*} Recall identity \eqref{BellPartition} to conclude, for $j\geq 1$, 
\begin{align}
(2j)!\sum_{\boldsymbol{\pi}\in E(2j)} \frac{\kappa_{\boldsymbol{\pi}}}{y_{\boldsymbol{\pi}}} p^{2j-m_2}&=(2j)!p^{2j}\sum_{\boldsymbol{\pi}\in E(2j)} \frac{1}{y_{\boldsymbol{\pi}}}\kappa_2^{m_2}\kappa_4^{m_4}\cdots \kappa_{2k}^{m_{2k}}p^{-m_2}\\
&=(2j)!p^{2j}\sum_{\boldsymbol{\pi}\in E(2j)}\frac{1}{y_{\boldsymbol{\pi}}}\kappa^{m_2}\kappa_4^{m_4}\cdots \kappa_{2k}^{m_{2k}}\\
&=p^{2j}B_{2j}(0, \kappa, 0,\kappa_4, 0,\ldots, 0,\kappa_{2j}),\label{MomentMethod2}
\end{align} where $\kappa=\kappa_2/p$.  The generating function \eqref{BellGen} for the complete Bell polynomials implies
\begin{align*}
\exp\Big(\sum_{\ell=1}^{\infty} x_{\ell} \frac{t^{\ell}}{\ell !}\Big)+\exp\Big(\sum_{\ell=1}^{\infty} x_{\ell} \frac{(-t)^{\ell}}{\ell!}\Big)=2\sum_{j=0}^{\infty} B_{2j}(x_1, x_2, \ldots, x_{2j})\frac{ t^{2j}}{(2j)!}
\end{align*} Setting $x_1=x_2=\cdots=x_{2j-1}=0$ in the above expression and then simplifying yields the identity
\begin{align*}
\sum_{j=0}^{\infty} B_{2j}(0, x_2,0, \ldots,0, x_{2j})\frac{ (pt)^{2j}}{(2j)!}=\exp\Big(\sum_{\ell=1}^{\infty} x_{2\ell} \frac{(pt)^{2\ell}}{(2\ell)!}\Big).
\end{align*} Setting $x_2=\kappa$ and $x_i=\kappa_i$ for $i>2$ in the above expression and then appealing to \eqref{MomentMethod2} implies
\begin{align}
\sum_{j=0}^{\infty}c_{2j}(p) \frac{t^{2j}}{(2j)!}&=\sum_{j=0}^{\infty} B_{2j}(0,\kappa,0, \ldots,0, \kappa_{2j})\frac{ (pt)^{2j}}{(2j)!} \notag\\
&=\exp\Big(\frac{\kappa p^2t^2}{2}+\sum_{\ell=2}^{\infty} \kappa_{2\ell} \frac{(pt)^{2\ell}}{(2\ell)!}\Big)\notag\\
&=\exp\Big(\frac{\kappa p^2 t^2}{2}-\frac{\kappa_2p^2 t^2}{2}+\sum_{\ell=1}^{\infty} \kappa_{2\ell} \frac{(pt)^{2\ell}}{(2\ell)!}\Big)\notag\\
&=\exp\Big(\frac{\alpha p^2t^2}{2}-\frac{\kappa_2p^2 t^2}{2}\Big) \exp\Bigg(\frac{ K(pt)+K(-pt)  }{2}   \Bigg)\notag\\
&=\exp\Big(\frac{1}{2}p(1-p)\sigma^2 t^2\Big)\sqrt{M(pt)M(-pt)}\label{MomentMethod3}
\end{align} in a neighborhood of $t=0$, where $M(t)$ and $K(t)$ denote the moment and cumulant generating functions of $X$, respectively.  If $X$ is a normal random variable with mean $\mu$ and variance $\sigma^2$, then the moment generating function of $X$ is given by $M(t)=\exp\big( \mu t+ \frac{1}{2}\sigma^2t^2\big)$ and \eqref{MomentMethod3} implies that
\begin{align*}
\sum_{j=0}^{\infty}c_{2j}(p) \frac{t^{2j}}{(2j)!}=\exp\Bigg(\frac{1}{2}p\sigma^2 t^2\Bigg)
\end{align*} in a neighborhood of $t=0$. The method of moments and \eqref{MomentMethod1} imply Theorem \ref{MainTheorem1} since the moment generating function for the sum of two independent random variables is given by the product of their moment generating functions.


\subsection{Proof of Theorem \ref{MainTheorem2}}

If $X$ is a symmetric sub-gaussian random variable, then $M(t)=M(-t)$ and \eqref{MomentMethod3} implies 
\begin{align*}
\sum_{j=0}^{\infty}c_{2j}(p) \frac{t^{2j}}{(2j)!}=\exp\Bigg(\frac{1}{2}p(1-p)\sigma^2 t^2\Bigg)M(pt).
\end{align*} The moment generating function of $pX$ converges for all $t\in \mathbb{R}$ since $pX$ is a sub-gaussian random variable with mean zero \cite[Prop.~2.5.2]{Vershynin}.  The right hand side is the moment generating function for the sum of $pX$ with an independent normal with mean zero and variance $p(1-p)\sigma^2$. The corresponding distribution is therefore determined by its moments. The method of moments and \eqref{MomentMethod1} now conclude the proof of Theorem \ref{MainTheorem2}.


\section{Proof of Theorem \ref{MainTheorem3}}\label{Section6}

Suppose $X_1, X_2, \ldots, X_n\sim X$ are i.i.d.~random variables defined on $(\Omega, \mathcal{F}, \mathbb{P}_{\Omega})$ in which $X$ is a sub-gaussian random variable with mean zero and variance $\sigma^2$.   The same reasoning of Section \ref{Section5} implies that for all $j\geq 1$, 
\begin{align*}
\lim_{n\to \infty} \mathbb{E}[W_n(\boldsymbol{s})]^j=j!\sum_{\boldsymbol{\pi}\in P_0(j)} \frac{ \kappa_{\boldsymbol{\pi}}}{y_{\boldsymbol{\pi}}}  p^{j-m_2}.
\end{align*} Set $\kappa=\kappa_2/p$ and apply \eqref{BellPartition} to conclude, for $j\geq 1$,
\begin{align*}
j!\sum_{\boldsymbol{\pi}\in P_0(j)} \frac{ \kappa_{\boldsymbol{\pi}}}{y_{\boldsymbol{\pi}}}  p^{j-m_2}&=j!p^j\sum_{\boldsymbol{\pi}\in P_0(j)} \frac{1}{y_{\boldsymbol{\pi}}} \kappa_2^{m_2}\kappa_3^{m_3}\cdots \kappa_j^{m_j} p^{-m_2}\\
&=j!p^j\sum_{\boldsymbol{\pi}\in P_0(j)} \frac{1}{y_{\boldsymbol{\pi}}} \kappa^{m_2}\kappa_3^{m_3}\cdots \kappa_j^{m_j}\\
&=p^j B_j(0, \kappa, \kappa_3, \ldots, \kappa_j).
\end{align*} The generating function \eqref{BellGen} for the complete Bell polynomials implies
\begin{align*}
\sum_{j=0}^{\infty} B_j(0, \kappa, \kappa_3, \ldots, \kappa_j) \frac{(pt)^j}{j!}&=\exp\Big( \frac{1}{2}\kappa p^2t^2+\sum_{\ell=3}^{\infty} \kappa_{\ell} \frac{(pt)^{\ell}}{\ell !}\Big)\\
&=\exp\Big( \frac{1}{2}\kappa p^2t^2-\frac{1}{2}\kappa_2 p^2t^2+\sum_{\ell=2}^{\infty} \kappa_{\ell} \frac{(pt)^{\ell}}{\ell !}\Big)\\
&=\exp\Big( \frac{1}{2}\kappa p^2t^2-\frac{1}{2}\kappa_2 p^2t^2\Big)\exp\Big(K(pt)\Big)\\
&=\exp\Big(\frac{1}{2}p(1-p)\sigma^2 t^2\Big)M(pt)
\end{align*} in a neighborhood of $t=0$, where $M(t)$ and $K(t)$ denote the moment and cumulant generating functions of $X$, respectively. The method of moments concludes the proof of Theorem \ref{MainTheorem3}.


\section{Proof of Theorem \ref{MainTheorem4}}\label{Section7}

Suppose $X$ is a random variable defined on  $(\Omega, \mathcal{F}, \mathbb{P}_{\Omega})$ which has a non-zero mean and admits a moment generating function. Let $X_1, X_2, \ldots, X_n\sim X$ be i.i.d.~ random variables. Denote $V_n(\boldsymbol{s})=W_n(\boldsymbol{s})/(\mu \sqrt{n})$ for brevity. Minkowski's inequality \cite[p.~242]{Billingsley} and the fact that $X_1, X_2, \ldots, X_n\sim X$ are identically distributed imply
\begin{align*}
\Big(\mathbb{E}_{\Omega}|V_n(\boldsymbol{s})|^j\Big)^{1/j}&=\mu^{-1}n^{-3/2}\Big(\mathbb{E}_{\Omega}|\sigma_1X_1+\sigma_2X_2+\cdots +\sigma_nX_n|^j\Big)^{1/j}\\
&=\mu^{-1}n^{-3/2} \norm{ \sigma_1X_1+\sigma_2X_2+\cdots +\sigma_nX_n}_j\\
&\leq \mu^{-1}n^{-3/2}\norm{X}_j \sum_{k=1}^n \sigma_k.
\end{align*} The Cauchy-Schwarz inequality yields the inequality
\begin{align*}
\Big(\sum_{k=1}^n \sigma_k\Big)^2=\Big(\sum_{k=1}^n 1\cdot \sigma_k\Big)^2\leq n \sum_{k=1}^n \sigma_k^2=n\p_2(\boldsymbol{s}).
\end{align*} Appealing to the inequality $\p_2(\boldsymbol{s})=\p_2(\boldsymbol{\lambda})=2m\leq n(n-1)$ now implies that there exists a constant $c$, which is independent of $n$, for which 
\begin{align}
\mathbb{E}_{\Omega}|V_n(\boldsymbol{s})|^j\leq \big(\mu^{-1}n^{-3/2}\big)^j\norm{X}_j^j\big(\sqrt{2nm}\big)^j\leq c.\label{UniformEnergy}
\end{align} Therefore, $\mathbb{E}_{\mathcal{G}}\mathbb{E}_{\Omega} |V_n(\boldsymbol{\lambda})|^j$ is finite and the Fubini-Tonelli theorem  ensures
\begin{align}
\mathbb{E}\big[V_n(\boldsymbol{\lambda})\big]^j=\mathbb{E}_{\mathcal{G}} \mathbb{E}_{\Omega} \big[V_n(\boldsymbol{\lambda})\big]^j,\label{FubiniEnergy}
\end{align} in which $\mathbb{E}[\cdot]$ denotes expectation with respect to the product measure $\mathbb{P}_{\mathcal{G}}\otimes \mathbb{P}_{\Omega}$.

Proposition \ref{Proposition:AsymptoticEnergy} and the uniform boundedness of the variables $\mathbb{E}_{\Omega} |V_n(\boldsymbol{s})|^j$ now allow us to compute the limit for the total expectation of $[V_n(\boldsymbol{s})]^j$, which we now highlight. Define $E_j=\{ G\,:\,  \big(\mu^{-1}n^{-3/2}\big)^jf_j(\boldsymbol{s})=\alpha_j\}$, where $\alpha_j=\alpha_j(n)$ is a real number to be chosen momentarily. Relation \eqref{FubiniEnergy} implies
\begin{align*}
\mathbb{E}[V_n(\boldsymbol{s})]^j&=\mathbb{E}_{\mathcal{G}}\big[\big(\mu^{-1}n^{-3/2}\big)^j f_j(\boldsymbol{s})\big]\\
&=\int_{\mathcal{G}}    \big(\mu^{-1}n^{-3/2}\big)^j f_j(\boldsymbol{s})    \,d\mathbb{P}_{\mathcal{G}}\\
&=\alpha_j \mathbb{P}_{\mathcal{G}}( E_j)+\int_{E_j^c} \big(\mu^{-1}n^{-3/2}\big)^j f_j(\boldsymbol{s})\,d\mathbb{P}_{\mathcal{G}}.
\end{align*}  The inequality $\big(\mu^{-1}n^{-3/2}\big)^j| f_j(\boldsymbol{s})|\leq \mathbb{E}_{\Omega} |V_n(\boldsymbol{s})|^j,$ together with \eqref{UniformEnergy}, implies that
\begin{align*}
\Big\vert\int_{E_j^c}\big(\mu^{-1}n^{-3/2}\big)^j f_j(\boldsymbol{s})\,d\mathbb{P}_{\mathcal{G}}\Big\vert\leq \int_{E_j^c}  \mathbb{E}_{\Omega} |V_n(\boldsymbol{s})|^j\,d\mathbb{P}_{\mathcal{G}}\leq c \mathbb{P}_{\mathcal{G}}(E_j^c).
\end{align*} Therefore,
\begin{align*}
\lim_{n\to \infty} \mathbb{E}[V_n(\boldsymbol{s})]^j=\lim_{n\to \infty}\Big( \alpha_j \mathbb{P}_{\mathcal{G}}( E_j)\Big)
\end{align*} provided $\mathbb{P}_{\mathcal{G}}(E_j)\to 1$ as $n\to \infty$ so that $\mathbb{P}_{\mathcal{G}}(E_j^c)\to 0$ as $n\to \infty$.  We now choose $\alpha_j$ according to Proposition \ref{Proposition:AsymptoticEnergy} to conclude
\begin{align}
\lim_{n\to \infty} \mathbb{E}\big[V_n(\boldsymbol{s})\big]^j=\Big( \frac{8\sigma_p}{3\pi}\Big)^j.\label{MomentMethodEnergy1}
\end{align} Finally, the series 
\begin{align*}
\sum_{j=0}^{\infty} \Big( \frac{8\sigma_p}{3\pi}\Big)^j\frac{t^j}{j!}=\exp\Big( \frac{8\sigma_p}{3\pi}t  \Big)
\end{align*} corresponds to the moment generating function of a point mass at $8\sigma_p/(3\pi)$. The method of moments concludes the proof of Theorem \ref{MainTheorem4}.


\section{Closing Remarks and Open Questions}\label{Section8}

There are several possible paths to take with this project. However, none of these paths appear to be particularly easy. We argue, without the intention of undermining its potentital for difficulty, that a natural step forward is to  replace the weights appearing in \eqref{EigenSequence} with the eigenvalues for the Laplacian and signless Laplacian of a random graph \cite[Ch.~7]{Cvetkovic}. We recall that the \emph{Laplacian} of a simple graph $G$ of order $n$ is the symmetric $n\times n$ matrix $L(G)$ defined by
\begin{align*}
L(G)=D(G)-A(G),
\end{align*} where $D(G)$ denotes the \emph{degree matrix of $G$} defined by setting $[D(G)]_{ii}$ equal to the degree of vertex $i$ and $[D(G)]_{ij}=0$ when $i\neq j$. The \emph{signless Laplacian} of a simple graph $G$ of order $n$ is the symmetric $n\times n$ matrix $Q(G)$ defined by
\begin{align*}
Q(G)=D(G)+A(G).
\end{align*} If $G$ is a random graph, then the upper-triangular entries of $L(G)$ and $Q(G)$ satisfy the hypotheses of the theorems in Section \ref{Section3}. Unfortunately, the vertex degrees of a random graph, while not highly correlated, are not independent. The theorems of Section \ref{Section3}, therefore, do not apply. A result due to Bryc, Dembo and Jiang on spectral measures of large Markov matrices  \cite[Thm.~1.3]{Bryc}, however, can likely be adapted to handle this new situation. Lastly, we remark that the Laplacian and signless Laplacian matrices are positive semi-definite and therefore have non-negative eigenvalues. This motivates the following.

\begin{problem}
Find analogs of Theorems \ref{MainTheorem3} and \ref{MainTheorem4} for the Laplacian and signless Laplacian spectra of  Erd\H{o}s-R\'enyi-Gilbert graphs. 
\end{problem}

Perhaps the most natural question to consider pertains to the assumptions imposed on $X_1, X_2, \ldots, X_n$. Ultimately, the sub-gaussian assumption in Theorems \ref{MainTheorem2} and \ref{MainTheorem3} is needed to ensure the partial moments of $W_n(\boldsymbol{a})$ are uniformly bounded over all graphs. This uniform boundedness is crucial for evaluating the limit of $\mathbb{E}[W_n(\boldsymbol{a})]^j$. The authors leave it as an interesting task to try and drop the sub-gaussian assumption in these theorems.

\begin{problem}
Can we relax the sub-gaussian condition imposed on the random variables $X_1, X_2, \ldots, X_n$ appearing in Theorems \ref{MainTheorem2} and \ref{MainTheorem3}?
\end{problem}

\medskip\noindent\textbf{Acknowledgments.} The authors thank Stephan Ramon Garcia and Ken McLaughlin for useful comments and suggestions in the preparation of this manuscript.


\end{document}